\documentclass[12pt]{amsart}

\usepackage[
colorlinks = true, 
citecolor = blue, 
pagebackref = true
pdfauthor = {Cook--Hughes}, 
pdftitle = \text{Cook--Hughes - Bounds for Lacunary maximal functions given by Birch--Magyar averages}
]{hyperref}

\usepackage{amsfonts,amsmath,amsthm,amssymb}
\usepackage{mathrsfs}

\usepackage[paper=a4paper,margin=1in]{geometry}% http://ctan.org/pkg/geometry
\newgeometry{top=1in,bottom=1in,right=1in,left=1in}

\newcommand*{\doi}[1]{\href{http://dx.doi.org/#1}{doi: #1}}

\newcommand{\R}{\ensuremath{\mathbb{R}}}
\newcommand{\Z}{\ensuremath{\mathbb{Z}}}
\newcommand{\C}{\ensuremath{\mathbb{C}}}
\newcommand{\N}{\ensuremath{\mathbb{N}}}
\newcommand{\T}{\ensuremath{\mathbb{T}}}

\newcommand{\1}{\ensuremath{{\mathbf{1}}}}
\newcommand{\bfa}{\ensuremath{\mathbf{a}}}
\newcommand{\bfq}{\ensuremath{\mathbf{q}}}

\newcommand{\<}{\ensuremath{\lesssim}}

\newcommand{\eps}{\ensuremath{\varepsilon}}
\newcommand{\la}{\ensuremath{\lambda}}

\newcommand{\LL}{\ensuremath{\mathcal{L}}}

\newcommand{\what}{\ensuremath{\widehat{\omega_{\la_l}}}}

\newcommand{\ds}{\ensuremath{\widetilde{d\sigma}}}

\newcommand{\eq}{\begin{equation}}
\newcommand{\ee}{\end{equation}}

\newcommand{\p}{\ensuremath{\mathfrak{p}}}

\newcommand{\oddprime}{\mathfrak{p}}
\newcommand{\form}{Q}

\newtheorem{thm}{Theorem}
\newtheorem{lemma}{Lemma}
\newtheorem{prop}{Proposition}
\newtheorem{conj}{Conjecture}
\newtheorem{cor}{Corollary}

\newtheorem*{thma}{Theorem A (\cite{MSW})}
\newtheorem*{thmb}{Theorem B (\cite{Akos})}

\theoremstyle{definition}

\theoremstyle{remark}
\newtheorem{rem}{Remark}

\title{Bounds for Lacunary maximal functions given by Birch--Magyar averages}

\author{Brian Cook}
\address{
    School of Mathematics
	\\	Kent State University
	\\	USA
}
\email{briancookmath@gmail.com}

\author{Kevin Hughes}
\address{
    School of Mathematics
	\\	The University of Bristol
	\\	Howard House
	\\	Queens Avenue
	\\	Bristol, BS8 1TW
	\\	UK
	\\ and the Heilbronn Institute for Mathematical Research, Bristol, UK
}
\email{khughes.math@gmail.com}

\begin{document}
% \title{Discrete maximal functions along lacunary sequences}
% \author{Brian Cook \quad Kevin Hughes}

\begin{abstract}
% We provide a result concerning discrete maximal functions over lacunary sequences on algebraic hypersurfaces defined by sufficiently non-singular integral forms of degree at least two. 
We obtain positive and negative results concerning lacunary discrete maximal operators defined by dilations of sufficiently nonsingular hypersurfaces arising from Diophantine equations in many variables. 
Our negative results show that this problem differs substantially from that of lacunary discrete maximal operators defined along a nonsingular hypersurface. 
Our positive results are improvements over bounds for the corresponding full maximal functions which were initially studied by Magyar. 

In order to obtain positive results, we use an interpolation technique of the second author to reduce problem to a maximal function of main terms. 
The main terms take the shape of those introduced in work of the first author, which is a more localized version of the main terms that appear in work of Magyar. 
The main ingredient of this paper is a new bound on the main terms near $\ell^1$. 
For our negative results we generalize an argument of Zienkiewicz. 
\end{abstract}

% \date{}

\maketitle

%%%%%%%%%%%%%%%%%%%%%%%%%

% ---
\section{Introduction}
% ---

% --
\subsection{Background}
% --

The discrete spherical averages are defined for a function $f$ on the integer lattice $\Z^n$ as 
\[
S^{(n)}_\la f(y) 
= 
\frac{1}{R_n(\la)}\sum_{|x|^2=\la} f(y-x)
\]
where $|x|^2=x_1^2+x_2^2+...+x_n^2$ and $R_n(\la)=\#\{x\in\Z^n:|x|^2=\la\}$. 
Magyar \cite{Mag} and  Magyar, Stein, and Wainger \cite{MSW} considered the question of  $\ell^p$-boundedness of the discrete spherical maximal operators 
\begin{equation}\label{1.1}
S^{(n)}_*f(y)=\sup_{\la\in\mathbb{N}} |S_\la f(y)|.
\end{equation}
A complete result on this is given in the latter work with a subsequent restricted weak-type endpoint result given by \cite{Ionescu}.

\begin{thma}\label{msw}
The operator $S^{(n)}_*$ is bounded on $\ell^p$ if and only if $n\geq 5$ and $p > \frac{n}{n-2}$.
\end{thma}

While this result cannot be improved, several interesting, related problems - concerning scenarios where the definition \eqref{1.1} is modified in some way - remain open. 
One difficult problem concerns the case when $n=4$ where the supremum taken only over odd integers $\la>0$; see \cite{Magyar:distribution,Hughes:thesis,Hughes:improving} for more information. 
The focus of this paper is on another problem which asks for the correct range of $\ell^p$-boundedness for maximal operators obtained from restricting the supremum in \eqref{1.1} to a fixed subsequence. 
To be precise, given a subsequence $\mathcal{L} = \{\la_1,\la_2,...\}\subset\mathbb{N}$, we are interested in the discrete spherical maximal function over $\mathcal{L}$ defined as
\begin{equation}\label{1.2}
S^{(n)}_{\mathcal{L},*}f(y) 
:= 
\sup_{\la\in\LL} |S_\la f(y)|.
\end{equation}

The motivation for these operators lies in Euclidean harmonic analysis where they have been extensively studied. 
In particular, a result of Calder\'on \cite{Calderon}, and independently of Coifman--Weiss \cite{CW}, states that the maximal function of Euclidean spherical averages over a lacunary subsequence of $\R$ is bounded on $L^p(\R^n)$ for all $1<p \leq \infty$ when $n \geq 2$. 
Recall that a sequence $\mathcal{L} = \{\lambda_1 < \lambda_2 < \dots \} \subset \R_{>0}$ is \emph{lacunary} if there is a constant $c>1$ such that $\la_{i+1}/\la_i>c$ for all $i$. 
This is an improvement over the boundedness of Stein's full spherical maximal function which is bounded on $L^p(\R^n)$ precisely in the range $p > \frac{n}{n-1}$ and $n \geq 2$; see \cite{Stein:spherical,Bourgain:circular}. 
The main conjecture in this Euclidean setting is that the lacunary spherical maximal function is weak-type (1,1). 
This conjecture remains open despite much interest. 
For instance, see \cite{Christ_weak11,STW_pointwise_lacunary,CK_lacunary}.

In analogy with these Euclidean harmonic analysis results, we expect that a discrete spherical maximal function $S^{(n)}_{\mathcal{L},*}$ over a lacunary subsequence $\mathcal{L} \subset \N$ is bounded on a larger range of $\ell^p(\Z^n)$ subspaces than the range of the full maximal function given by Theorem A. 
Naively, one would expect that $S^{(n)}_{\mathcal{L},*}$ is bounded on $\ell^p(\Z^n)$ for all $p>1$. 
However, Zienkiewicz showed that this may fail. 
Initially unaware of these limitations, the authors showed that $S^{(n)}_{\mathcal{L},*}$ may indeed be bounded on a larger range of $\ell^p(\Z^n)$ spaces for certain types of lacunary subsequences; see \cite{Hughes:sparse,Cook:sparse} for more detail. 
Combining ideas from the authors' works, \cite{KLM} recently obtained a result which only requires $\mathcal{L}$ to be lacunary. 
In particular \cite{KLM} showed that for each lacunary sequence the lacunary discrete spherical maximal function $\mathcal{L} \subset \N$ lacunary, $S^{(n)}_{\mathcal{L},*}$ is bounded on $\ell^p(\Z^n)$ for $p>\frac{n-2}{n-3}$ and $n \geq 5$. 

The discrete spherical maximal function and its restriction to subsequences lie in a more general context introduced by Magyar. 
In \cite{Akos} Magyar extended Theorem~A of Magyar--Stein--Wainger to sufficiently nice positive definite hypersurfaces; we will describe precisely what we mean below. 
Magyar only recorded the $\ell^2(\Z^n)$ boundedness of the associated maximal functions. 
However his methods allow one to obtain $\ell^p(\Z^n)$ boundedness for a range of $p$ slightly below $\ell^2$. 
Unfortunately, determining the sharp range of $\ell^p(\Z^n)$ boundedness in the range of $p$ or $n$ is far too difficult a problem with current methods. 
To convey its difficulty note that one first requires a complete resolution of Waring's problem.
Instead, the goal of the current work is to obtain a general result for lacunary discrete maximal functions associated to Diophantine equations in many variables. 
A particular case of our main result yields a distinct proof of the result in \cite{KLM}. 

\begin{rem}
The above problems and results - as well as those below - concern averages \emph{transverse} to a given hypersurface, and differ substantially from the indefinite case where one averages \emph{along a fixed hypersurface}, e.g. Bourgain's averages along the squares in \cite{Bourgain:squares,Boup>1}. 
In the indefinite case, the full maximal function along the hypersurface is equivalent - up to a factor of a power of 2 - to the maximal function along the dyadic sequence $2^j$ for $j \in \N$ along the hypersurface. 
For instance, the second author \cite{Cook:Birch} previously obtained $\ell^p(\Z^n)$ boundedness of the corresponding lacunary maximal functions for all $p>1$. 

The endpoint behavior at $p=1$ is subtle in all of these problems. 
In some instances, see \cite{UZ,Christ_weak_discrete,Mirek_weak11}, the weak-type (1,1) estimate holds while in other cases, see \cite{BM_divergent_squares,LaVictoire_universally_bad}, the weak-type (1,1) estimate fails. 
We will see below that the behavior of the maximal functions considered here (which are given by transverse dilations of a hypersurface) possess significantly more subtle phenomena than that of maximal operators defined along a hypersurface.
\end{rem}

%%%%%%%%%%%%%%%%%%%%%%%%%
% \bigskip
% --
\subsection{The main results} 
% --

Our results lie in the Magyar's framework in \cite{Akos} which generalized Theorem~A to averages over more general families of surfaces given by Diophantine equations in many variables. 
We describe this framework before stating our results. 

Let $Q$ be a homogeneous integral  polynomial (i.e., an integral form)  of degree $k$ in $n$ variables and define the Birch rank of $Q$, denoted $\mathcal{B}(Q)$, to be the co-dimension of the complex singular locus  $\{z\in\C^n:\partial_{z_1}Q(z)=...=\partial_{z_n}Q(z)=0\}$ of $Q$.  For a given test function $\psi$ with $0\leq\psi\leq1$ we define the counting function\[
r_{Q,\psi}(\la):=\sum_{Q(x)=\la} \psi(x/\la^{1/k}).\]
From \cite{Bi} we have that  
\[
0\leq r_{Q,\psi}(\la)\leq\sum_{Q(x)=\la}\1_{x\in\la^{1/k} \text{supp}(\psi)}\<\la^{n/k-1}\]
on the condition $\mathcal{B}(Q)>(k-1)2^k$. 
An integral form $Q$ satisfying this rank condition, together with the condition that $Q(x)=1$ has a nonsingular solution $x\in \R^n$ on the support of a given test function $\psi$, is said to be $\psi$-$regular$. 
If $Q$ is $\psi$-regular for a given $\psi$, as is  shown in \cite{Akos}, there is an infinite arithmetic progression $\Gamma$ of \emph{regular values} such that $\la\in\Gamma$ has 
\[
\la^{n/k-1} 
\lesssim 
r_{Q,\psi}(\la).
\]
The precise set of regular values is actually independent of $\psi$ (up to the first few terms possibly), although the lower bound here depends on the existence of a nonsingular solution.
Henceforth, for a $\psi$-regular form $Q$, let $\Gamma$ denote such an infinite arithmetic progression. 

For $\la\in \Gamma$ we define the averages
\[
A^{(Q,\psi)}_\la f(y)=\frac{1}{r_{Q,\psi}(\la)}\sum_{Q(x)=\la}\psi(x/\la^{1/k}) f(y-x).
\]
These averages appear more general than the averages considered in \cite{Akos}, but the methods in \cite{Akos} extend to include these averages. 
Our reason for including the cutoff is that we wish to include cases where the surface in $\R^n$ given by the equation $Q(x)=1$ is not compact. 

For the full maximal function 
\[
A^{(Q,\psi)}_*f(y)=\sup_{\la\in\Gamma} |A^{(Q,\psi)}_\la f(y)|
\]
over the variety $\{ Q = 1\}$ we have the following result due to Magyar. 
\begin{thmb}
Let $Q$ be a $\psi$-regular integral form of degree $k>1$ with a sequence of regular values $\Gamma$. If $A^{(Q,\psi)}_*$ is defined as above then we have 
\[
\|A^{(Q,\psi)}_*\|_{\ell^2\to\ell^2} 
< 
\infty.
\]
\end{thmb}

The correct exponent for $\ell^p$-boundedness in this result remains open. 
Considering the maximal function of a delta function in $\Z^n$ shows that the bound on the exponent in Theorem A is sharp. 
Moreover, the same example shows that $A^{(Q,\psi)}_*$ is unbounded on $\ell^p(\Z^n)$ for $p \leq \frac{n}{n-k}$. 
Most likely, $p > \frac{n}{n-k}$ is the sharp range of $\ell^p(\Z^n)$ boundedness for sufficiently large $n$ with respect to the degree $k$ of the form $Q$. 
As noted in \cite{Hughes:improving}, and independently by the first author, the range of $\ell^p(\Z^n)$ boundedness in Theorem~B extends slightly to $p$ below $2$, but nowhere near to the critical exponent of $n/(n-k)$. 
% In particular, $A^{(Q,\psi)}_*$ is bounded on $\ell^p(\Z^n)$ for $\max\{\frac{\mathcal{B}(Q)/((d-1)2^{d-1})}{\mathcal{B}(Q)/((d-1)2^{d-1})-1},\}<p<\infty$ which is $\max\{\frac{\mathcal{B}(Q)}{\mathcal{B}(Q)-(d-1)2^{d-1}},\}<p\leq\infty$ and $\mathcal{B}(Q)>(d-1)2^d$. 
Substantial improvements to the range of $\ell^p(\Z^n)$ boundedness for $A^{(Q,\psi)}_*$ remain intractable at this time due to the present limitations of the circle method. 
% Some minor improvements can be made at this level of generality (see \cite{Kevin2} for such a result), but the current level of understanding surrounding certain aspects of the circle method used in studying the counting functions make this an intractable problem. 
Instead we are interested in bounds for lacunary maximal operators
\[
A^{(Q,\psi)}_{\LL,*}f(y)
:= 
\sup_{\la_l\in\LL} |A^{(Q,\psi)}_{\la_l}f(y)| = 
\sup_{\la_l\in\LL}  |\frac{1}{r_{Q,\psi}(\la_l)} \sum_{Q(x)=\la_l}\psi(x/\la_l^{1/k})f(y-x)|
\]
where $\LL$ is a lacunary subsequence of $\Gamma$. 

For an integral polynomial $Q$ in $n$ variables we define the normalized Weyl sums 
\[
F_q^{(Q)}(a,\bfa) 
= 
q^{-n}\sum_{s\in Z_q^n}e(Q(s)\,a/q+s\cdot \bfa/q)
\]
where $a\in Z_q$ and $\bfa\in Z_q^n$. 
Here we introduce the notations $e(z)=e^{2\pi iz}$, $Z_q=\Z/q\Z$, and $U_q=Z_q^*$ (with the understanding that $U_1=Z_1=\{0\}$, allowing the notation to keep track of any major arc around $0$). 
For the statement of our main result we need the quantity 
\begin{equation}\label{eq:alpha}
\alpha_Q 
:= 
\sup\{\beta \geq 0 : \sup_{a\in U_q}\sup_{\bfa\in Z^n_q}|F^{(Q)}_q(a,\bfa)|\< q^{-\beta}\}.
\end{equation}
An important property of $Q$ being regular is that it forces the bound $\alpha_Q>2$; this is necessary in our approach. 
In Corollary~2 of \cite{Akos} Magyar proves that $\alpha_Q \leq \frac{\mathcal{B}(Q)}{2^{k-1}(k-1)}$ for $\mathcal{B}(Q) > 2^k(k-1)$. 
We make this assumption on the degree and dimension implicitly throughout. 
This bound is not sharp in general; for instance, one may improve it significantly when $Q$ is a diagonal form. 

Using the approach in \cite{Hughes:restricted} one may prove that for any lacunary subsequence $\LL$ of $\Gamma$, the lacunary maximal function $A^{(Q,\psi)}_{\LL,*}$ is bounded on $\ell^p(\Z^n)$ for $p > \frac{\alpha_Q}{\alpha_Q-1}$. 
This improves upon the range implicit in \cite{Akos}. 
In this work we obtain the following further improvement. 
(Note that $\frac{2\alpha_Q-2}{2\alpha_Q-3} < \frac{\alpha_Q}{\alpha_Q-1}$ precisely when $\alpha_Q>2$.) 
% The supporting calculation:
% \[
% \frac{\alpha}{\alpha-1} > \frac{2\alpha-2}{2\alpha-3} 
% \leftrightarrow 
% \alpha(2\alpha-3) > (\alpha-1)(2\alpha-2)
% \leftrightarrow 
% 2\alpha^2-3\alpha > 2\alpha^2-4\alpha+2
% \leftrightarrow 
% \alpha > 2.
% \]

\begin{thm}\label{1}
Let $Q$ be a $\psi$-regular integral form of degree $k>1$  with a sequence $\Gamma$ of regular values. If  $\LL=(\la_l)_{l=1}^\infty\subseteq \Gamma$ is a lacunary sequence, then the maximal operators $A^{(Q,\psi)}_{\LL,*}$ are bounded on $\ell^p(\Z^n)$ for \(p > \frac{2\alpha_Q-2}{2\alpha_Q-3}\)
and the dimension $n$ is sufficiently large.
\end{thm}

% As a complement to this we prove that it is impossible to obtain $\ell^p$ boundedness for all $p>1$ by generalizing an argument of Zienkiewicz for spheres. 
In tandem with Theorem~\ref{1}, we generalize Zienkiewicz's probabilistic construction for spheres in \cite{Zienkiewicz} to show that there exists lacunary sequences $\LL$ for which  $A^{(Q,\psi)}_{\LL,*}$ is unbounded near $\ell^1(\Z^n)$. 
\begin{thm}\label{theorem:lowerbound}
Let $Q$ be a $\psi$-regular integral form of degree $k>1$  with a sequence $\Gamma$ of regular values. If $p \in (1, \frac{n}{n-1})$, then there exists a lacunary sequence $\LL$ of $\N$ such that $A^{(Q,\psi)}_{\LL,*}$ is unbounded on $\ell^p(\Z^n)$. 
\end{thm}

To give the reader a explicit sense of the range in Theorem~\ref{1} we consider a few specific examples and compare this result to previous results and conjectures.
\begin{itemize}
\item%[Spheres:] 
{\bf Spheres}: When $Q(x)=x_1^2+...+x_n^2 $, $n\geq5$, we can choose any test function $\psi$ which is identically one on the unit sphere in $\R^n$ to recover the spherical averages $S^{(n)}_{\la}$. 
The exponential sum in question is an $n$-fold product of one dimensional Gauss sums, giving that $\alpha_Q=n/2$. 
In turn, we have $\ell^p$-boundedness of the maximal functions $S^{(n)}_{\mathcal{L},*}$ along a lacunary sequence $\LL$ for
\(
p>\frac{n-2}{n-3},
\)
recovering the result in \cite{KLM}. 
% This is an improvement over $\frac{n}{n-2}$. 
Moreover, our result gives the same range of $p$ whenever $Q$ is a positive definite integral quadratic form since $\alpha_Q$ remains the same for such quadratic forms; see \cite{HH:quadratic} for a proof of these bounds. 
% Unfortunately we do not obtain a result for $n=4$. 

\item%[$k$-spheres:]
{\bf $k$-Spheres}: Let  $Q(x)=x_1^k+...+x_n^k$ for an $k\geq3$, where $n$ is sufficiently large in terms of $k$. %\footnote{Precisely how large $ n$ needs to be in terms of $k$ is an interesting question, but we make no attempt to address this at this point in time.} 
If $k$ is even, then the obvious choice for $\psi$ is a test function which is identically one on the unit surface $Q(x)=1$ in $\R^n$. 
If $k$ is odd, then we choose $\psi$ so that the restriction its support is contained in the positive orthant. 
Alternatively, one can also work with the expressions $Q(x)=|x_1|^k+\cdots+|x_n|^k$ and choose $\psi$ identically one on the $\R^n$ surface given by $Q(x)=1$.
In either case, a result of \cite{Steckin} gives $\alpha_Q=n/k$ which yields $\ell^p(\Z^n)$ boundedness for
\(
p>\frac{2n-2k}{2n-3k}
\) 
when $n \gtrsim k^2$. 
See \cite{ACHK} for the precise range of $n$ we may obtain here. 
The exponent $\frac{2n-2k}{2n-3k}$ is smaller than the exponent $\frac{n}{n-k}$ for which the full maximal operator $A^{(Q,\psi)}_{\N,*}$ is unbounded. 
\item%[Birch forms:] 
{\bf Birch--Magyar forms}: For a general regular form $Q$ of degree $k\geq3$, assuming that $Q$ is $\psi$-regular, Theorem \ref{1} gives the bound 
\[
p > \frac{\mathcal{B}(Q)-2(k-1)2^{k-2}}{\mathcal{B}(Q)-3(k-1)2^{k-2}}\,
\]
from the estimates in (\cite{Bi}, Lemma 5.4). 
Even in the nonsingular cases where $\mathcal{B}(Q)=n$ this fails to reach the exponent $\frac{n}{n-k}$. 
However, this represents an improvement over currently available bounds for the full maximal operator $A^{(Q,\psi)}_{*}$.
\end{itemize}

In all of the above cases, the exponents are larger than $\frac{n}{n-1}$ so that Theorem~\ref{1} does not conflict with Theorem~\ref{theorem:lowerbound}. 
Furthermore, Theorem~\ref{theorem:lowerbound} does not say that every lacunary sequence is unbounded on $\ell^p(\Z^n)$ for $1<p<\frac{n}{n-1}$. 
In fact the first author in \cite{Cook:sparse,Cook:Birch} showed that there exists lacunary sequences $\LL \subset \N$ for which $A^{(Q,\psi)}_{\LL,*}$ is bounded on $\ell^p(\Z^n)$ for all $p>1$ and $n \geq 5$ when $Q(x) := x_1^2+\cdots+x_n^2$. 
Surprisingly, the second author has formulated a precise conjecture regarding the $\ell^p(\Z^n)$ boundedness when $Q(x) := x_1^2+\cdots+x_n^2$; see the forthcoming work \cite{HWZ}. 

With the above results and examples in mind, we make the following two conjectures. 
\begin{conj}
Let $Q$ be a $\psi$-regular integral form of degree $k>1$ with a sequence $\Gamma$ of regular values. 
The maximal operator $A^{(Q,\psi)}_{\Gamma,*}$ is bounded on $\ell^p(\Z^n)$ for 
\(
p > \frac{n}{n-k} 
\)
when the dimension $n$ is sufficiently large with respect to the degree $k$. 
\end{conj}

\begin{conj}
Let $Q$ be a $\psi$-regular integral form of degree $k>1$ with a sequence $\Gamma$ of regular values. 
If $\LL=(\la_l)_{l=1}^\infty\subseteq \Gamma$ is a lacunary sequence, then the maximal operators $A^{(Q,\psi)}_{\LL,*}$ are bounded on $\ell^p(\Z^n)$ for 
\(
p > \frac{n}{n-1}
\)
when the dimension $n$ is sufficiently large with respect to the degree $k$. 
\end{conj}

%%%%%%%%%%%%%%%%%%%%%%
% \bigskip
% --
\subsection{Notation}
% --

\begin{itemize}
\item
We will write $f(\lambda) \lesssim g(\lambda)$ if there exists a constant $C>0$ independent of all $\lambda$ under consideration (e.g. $\lambda$ in $\N$ or in $\Gamma_{\form}$) such that 
\[
|f(\lambda)| \leq C |g(\lambda)|. 
\]
Furthermore, we will write $f(\lambda) \gtrsim g(\lambda)$ if $q(\lambda) \lesssim f(\lambda)$ while we will write $f(\lambda) \eqsim g(\lambda)$ if $f(\lambda) \lesssim g(\lambda)$ and $f(\lambda) \gtrsim g(\lambda).$
% \item
% There will be several $\epsilon$-losses that arise. 
% Rather than writing ``for all $\epsilon>0$ ...'' we will write $f(\lambda) \lessapprox g(\lambda)$ if $f(\lambda) \lesssim_\epsilon \lambda^{\epsilon} g(\lambda)$ for all $\epsilon>0$; the implicit constants are allowed to depend on $\epsilon$ but not on $\lambda$. 
% For instance the conclusion of Theorem~\ref{theorem:discrete_improving:Birch} may be more succintly written as 
% \[
% \|A_\lambda f\|_{\ell^{p'}(\Z^d)} 
% \lessapprox_{\form,p} \lambda^{-\eta_{\form,p}} \|f\|_{\ell^{p}(\Z^d)}.
% \]
\item
Subscripts in the above notations will denote parameters, such as the dimension $n$ or degree $k$ of a form $\form$, on which the implicit constants may depend. 
\item
$\T^n$ denotes the $n$-dimensional torus $(\R/\Z)^n$ identified with the unit cube $[-1/2,1/2]^n$. 
\item
$*$ denotes convolution on a group such as $\Z^n$, $\T^n$ or $\R^n$. 
It will be clear from context as to which group the convolution takes place.
\item
$e(t)$ will denote the character $e^{-2\pi it}$ for $t \in \R$, $\T$ or $\Z_q$.
\item
For a function $f: \Z^n \to \C$, its $\Z^n$-Fourier transform will be denoted $\hat{f}(\xi)$ for $\xi \in \T^n$. 
For a function $f: \R^n \to \C$, its $\R^n$-Fourier transform will be denoted $\widetilde{f}(\xi)$ for $\xi \in \R^n$. 
\item 
The inverse Fourier transform on $\Z^n$ will be denoted by $\mathscr{F}^{-1}$. 
\item
$\mu$ denotes the M\"obius function. 
\item
For $j \in \N$, set $I_j=[2^{j-1},2^j) \cap \Z$.
\end{itemize}

% --
\subsection*{Acknowledgements}
% --

The authors thank Stephen Wainger for discussions on this problem, and for his encouragement in pursuing it. 
This collaboration was supported by funding from the Heilbronn Institute for Mathematical Research.

%%%%%%%%%%%%%%%%%%%%%%%%%
% \bigskip
% ---
\section{Reduction to the maximal function of the main term}\label{section:reduction}
% ---

Here we give an outline of the argument for Theorem~\ref{1}. 
The main approach is in line with the argument in \cite{MSW} which goes as follows.
\begin{enumerate}
\item 
Find a suitable approximation of the surface measures by Fourier multipliers yielding main terms and an error term. Each Fourier multiplier corresponds to a distinct convolution operator with an associated maximal operator.
\item 
Deduce $\ell^p$-boundedness for the maximal operators that arise from the main terms by transference from a Euclidean maximal function inequality. 
\item 
Use a partial result (such as the main result of \cite{Mag}) to show that the maximal operator arising from the error term is bounded on the appropriate $\ell^p$ spaces.  
\end{enumerate} 
Lacunary maximal functions allow for greater flexibility in Step~3 which gives us more freedom in choosing the form of our main terms in the approximation of Step~1.
 
To proceed, fix a $\psi$-regular integral form $Q$ and a lacunary sequence $\LL := (\la_l)_{l=1}^\infty \subset \Gamma$ inside a sequence of regular values $\Gamma$ of the form $Q$ for the remainder of the paper. 
The dependence on  $\LL$, $Q$, and $\psi$ is suppressed from the notation. 
% We can also assume that the sequence of regular values $\Gamma$ is $\mathbb{N}$  without any loss of generality. 
Our starting point is to find a suitable approximation to the Fourier multipliers \[
\what(\xi) 
= 
\frac{1}{r(\la_l)}\sum_{Q(x)=\la_l} e(x\cdot\xi)
\]
of the shape $m_{\la_l}(\xi)+e_{\la_l}(\xi)$. 
However, we first state our interpolation inequality which reveals what we need from $e_{\la_l}$, and hence $m_{\la_l}$. 
\begin{lemma}\label{errors}
Suppose that each multiplier $\what = m_{\lambda_l}+e_{\lambda_l}$ decomposes into terms $m_{\lambda_l}$ and $e_{\lambda_l}$ for each $\lambda_l \in \LL$. 
Define the convolution operator
\begin{align*}
M_{\la_l} f 
& = 
\mathscr{F}^{-1}(\hat{f} \cdot m_{\la_l}),
\end{align*}
and the maximal operator
\begin{align*}
M_*f 
& = 
\sup_{l\geq 1} | M_{\la_l} f |.
\end{align*}
Assume that there exists some $\delta>0$ such that 
\[
\|e_{\la_l}\|_{L^\infty(\T^n)}
\<
\la_l^{-\delta}
\]
uniformly for all $l \in \N$. 
If $M_*$ is bounded on $\ell^{p_0}(\Z^n)$ for some $1 \leq p_0 \leq 2$, then $A_*$ is bounded on $\ell^p(\Z^n)$ for all $p>p_0$. 
\end{lemma}

\begin{rem}
We will use a similar statement over $\R^n$ which follows by an analogous proof. 
\end{rem}

\begin{proof}
Let $\what, m_{\lambda_l}, e_{\lambda_l}$ and $\delta$ be as in the statement of the Lemma. 
Write $A_{\lambda_l}, M_{\lambda_l}, E_{\lambda_l}$ for the convolution operators rescpectivly associated to the multipliers $\what, m_{\lambda_l}, e_{\lambda_l}$.
Recall that $I_j$ denotes the discrete interval $[2^{j-1},2^j) \cap \Z$. 
Also let $p_0$ denote the exponent of the lemma; that is, the exponent for which $M_*$ is bounded. 
Since $A_{\la_l} = M_{\la_l}+E_{\la_l}$, we have 
\[
\|\sup_{\la_l\in I_j } |A_{\la_l}f|\|_{\ell^p} 
\leq 
\|\sup_{\la_l\in I_j } |M_{\la_l}f|\|_{\ell^p} + \|\sup_{\la_l\in I_j } |E_{\la_l}f|\|_{\ell^p} 
\]
for each $p$ and all $f \in \ell^p$. 
Take $p_0<p\leq2$. 
Since we have assumed that the maximal operator $\sup_{\la_l\in I_j } |M_{\la_l}|$ is bounded on $\ell^{p_0}$, it suffices to show that the maximal operator $\sup_{\la_l\in I_j } |E_{\la_l}|$ is bounded on $\ell^p$ for $p>p_0$. 

There exists a constant $0<C<\infty$ depending on $\LL$ such that for each $j$ we have $\#(\LL \cap I_j) \leq C$ since $(\la_l)_{l=1}^\infty$ is a lacunary sequence.
This implies that 
\[
\|\sup_{\la_l\in I_j } |E_{\la_l}f|\|_{\ell^p} 
\leq 
\sum_{\la_l\in I_j} \|E_{\la_l}f\|_{\ell^p} 
\leq 
C \sup_{\la_l\in I_j} \|E_{\la_l}f\|_{\ell^p}
%= 
%C \sup_{\la_l \in I_j} \|A_{\lambda_l}f-M_{\lambda_l}f\|_{\ell^p}
\]
Therefore, it suffices to show that 
\[
\|E_{\la_l}f\|_{\ell^p} 
\lesssim 
\la_l^{-\eta}
\]
for some $\eta>0$ depending on $p$. 
The hypothesis $\|e_{\la_l}\|_{L^\infty(\T^n)}\<\la_l^{-\delta}$ is equivalent to $\|E_{\la_l}f\|_{\ell^2 \to \ell^2} \lesssim \la_l^{-\delta}$. 
By interpolation, it now suffices to show that $\|E_{\la_l}f\|_{\ell^{p_0} \to \ell^{p_0}} \lesssim 1$ uniformly for $\la_l \in \LL$. 

Young's inequality implies that $\sup_{\la_l \in \LL} \|A_{\la_l}\|_{\ell^p \to \ell^p} \leq 1$ for all $1 \leq p \leq \infty$ since each of the averages $A_{\la_l}$ is given by convolution with a probability measure. 
Since $M_*$ is bounded on $\ell^{p_0}$, we have 
\[
\sup_{\la_l \in \LL} \| M_{\la_l}f \|_{\ell^{p_0}} 
\leq 
\|M_*f\|_{\ell^{p_0}} 
\lesssim 
\|f\|_{\ell^{p_0}}.
\]
Therefore,
\[
\sup_{\la_l \in \LL} \| E_{\la_l}f \|_{\ell^{p_0}} 
=
\sup_{\la_l \in \LL} \| (A_{\la_l}-M_{\la_l})f \|_{\ell^{p_0}} 
\leq 
\sup_{\la_l \in \LL} \| A_{\la_l} \|_{\ell^{p_0}} + \sup_{\la_l \in \LL} \| M_{\la_l}f \|_{\ell^{p_0}} 
\lesssim 
\|f\|_{\ell^{p_0}}
\]
as desired. 
\end{proof}

In our method, the following form of the main term $m_\la$ is useful.  
Define the Fourier multipliers on $\T^n$ as 
\begin{equation}\label{2.1}
m_{\la_l,j,q}(\xi) 
:= 
\sum_{a \in U_q} \sum_{\bfa\in Z_q^n} F_q(a,\bfa) e(-a \la_l/q) \zeta(10^j(\xi-\bfa/q)) \ds(\la_l^{1/k}(\xi-\bfa/q))
\end{equation}
for all $q\in I_j=[2^{j-1},2^j)$. 
Here, the function $\zeta$ is a fixed smooth, non-negative function supported on the interval $[-1/5,/1,5]$ which is identically one on $[-1/10,1/10]$, and the measure $d\sigma$ is defined by 
\[
d\sigma(x) = \psi(x)\frac{d\mu(x)}{|\nabla Q(x)|}
\]
where $\mu$ the Euclidean surface measure on the Euclidean surface $\{ x \in \R^n: Q(x)=1\}$. 
Section~6 of \cite{Bi} shows that the density $\psi/|\nabla Q|$ is well behaved and that assuming the existence of a nonsingular solution in the support of $\psi$ gives that $\int d\sigma > 0$.

The particular shape \eqref{2.1} of the main term is from \cite{Cook:sparse}; see also \cite{Cook:Birch}. 
This main term relates to the decomposition used in \cite{Akos}, but more closely resembles the main terms of \cite{Boup>1}. 
Our main term has the advantage of being very localized near rational points $\bfa/q$. 
Note that for each $q \in [2^{j-1},2^j) \cap \Z$, the multipliers $\zeta(10^j(\xi-\bfa/q))$ are disjoint as $\bfa/q$ varies over $\bfa \in Z_q^n$. 
While this is important in \cite{MSW}, it is more important for us that the factor of $10^j$ is the same for all $q \in [2^{j-1},2^j) \cap \Z$. 

For each $\lambda \in \Gamma$ we define 
\[
m_{\la,j}(\xi) 
= 
\sum_{q\in I_j} m_{\la,j,q}(\xi).
\]
Our main term in the approximation is then
\[
m_{\la}(\xi) 
= 
\sum_{j=1}^\infty m_{\la,j}(\xi).
\]
Lemma~5 in \cite{Cook:sparse} yields the following $\ell^2(\Z^n)$-approximation for our averages. 
\begin{lemma}[$\ell^2$-approximation]\label{ell2:approximation}
There exists a $\delta>0$ depending on $Q$ such that 
\[
\|\what-m_{\la_l}\|_{L^\infty(\T^n)}
\< 
\la_l^{-\delta}.
\]
\end{lemma}

Applying Lemma~\ref{errors} and Lemma~\ref{ell2:approximation}, it remains to show that the operator $M_*$ is bounded on $\ell^p$ for $p > \frac{2\alpha_Q-2}{2\alpha_Q-3}$; that is, when $p$ satisfies the bound in Theorem \ref{1}. 
To accomplish this task, we will obtain bounds near $\ell^1$ and bounds on $\ell^2$ for each of the maximal operators $M_{*,j}$ with which to interpolate and sum up. 
This is the dominant approach for essentially all problems of this type (continuous or discrete). 

More precisely, we will use the following bound. 
\begin{lemma}[See \cite{Akos}]\label{lemma:l2approximation}
Uniformly for $j \in \N$, we have 
\[
\|M_{*,j}\|_{\ell^2\to\ell^2} 
\<_\beta 
2^{(2-\beta)j}
\]
for all $\beta < \alpha_Q$ (as defined in \eqref{eq:alpha}). 
\end{lemma}
This is proved in \cite{Akos}. 
Our main contribution is to give a new estimate of the following form. 
\begin{thm}\label{theorem:new}
For each $j \in \N$, we have 
\begin{equation}\label{estimate:new}
\|M_{*,j}\|_{\ell^p\to\ell^p} 
\<_{p} 
j^2 2^{j}
\end{equation}
for all $1<p<2$.
\end{thm}
To prove Theorem~\ref{1}, one interpolates the bounds from Lemma~\ref{lemma:l2approximation} and Theorem~\ref{theorem:new} and sums over $j$ to deduce the desired range of $p$. 
% This is carried out in Section~\ref{section:interpolation}
We carry this out momentarily after describing how we prove Theorem~\ref{theorem:new}. 

% This is the basic premise, but there are certain ripples to smooth out along the way to obtain the necessary estimates for $||M_{*,j}||_{\ell^p\to\ell^p}$ near $p=1$. 
In order to prove the latter estimates near $\ell^1$,  we use a M\"obius inversion argument in Section~\ref{section:Mobius} to give a suitable rearrangement of the terms $m_{\la_l,j, q}$ as a divisor sum of complete exponential sum analogues of the $m_{\la_l,j, q}$. 
After summing over $q\in I_j$ we end up with a manageable expression for the multipliers $m_{\la_l,j}$. 
The use of the complete sum analogues dates back to Bourgain \cite{Boup>1}, and  in the current context these  are considered in \cite{Cook:sparse}. 

The presence of the characters $e(-a \la_l/q)$ in \eqref{2.1} introduces several difficulties in comparison with \cite{Cook:Birch}. 
In Section~\ref{section:discrete} we deal with these issues by comparison with a certain Euclidean maximal function result. 
The Euclidean maximal function is bounded in Section~\ref{section:Euclidean} and is based on a result in \cite{DR86}. 
The results of Section~\ref{section:Euclidean} only  treat  the maximal averages along suitable tails of $\LL$, and thus the results of Section~\ref{section:discrete} do as well.
As we shall see, what constitutes `the tail' is reasonable. 
And, thanks to the lacunary condition, this allows us to treat the maximal averages along the entire lacunary sequence $\LL$ in order to conclude the proof of Theorem~\ref{1} in the ultimate section. 

\begin{proof}[Proof of Theorem~\ref{1}]
Since the full maximal function is bounded on $\ell^2(\Z^n)$, it suffices to consider only $1<p<2$.
We have the maximal operator $M_*$ corresponding to the multipliers $m_{\la_l}(\xi)$ and the maximal operators $M_{*,j}$ corresponding to the multipliers $m_{\la_l,j}(\xi)$. 
The triangle inequality implies that 
\[
||M_*||_{\ell^p\to\ell^p}
\leq 
\sum_{j=1}^\infty||M_{*,j}||_{\ell^p\to\ell^p},
\]
so it suffices to show the bound
\[
||M_{*,j}||_{\ell^p\to\ell^p}
\<_\epsilon
2^{(\epsilon-\nu)j}
\]
for some $\nu>0$ depending on $p$ when $\frac{2\alpha-2}{2\alpha -3}<p<2$ and all sufficiently small $\epsilon>0$. 

Select $1<p_0<2$ close to 1. 
Choose an $\epsilon>0$ sufficiently small and select $\beta$ such that $\alpha-\epsilon/2 \leq \beta \leq \alpha$. 
Interpolating the bound on $\ell^2$ from Lemma~\ref{errors} with the $\ell^{p_0}$ bound from Theorem~\ref{theorem:new} we find that 
\[
\|M_{*,j}\|_{\ell^p\to\ell^p}
\<
2^{(\epsilon-\nu)j}
\]
for all sufficiently small $\epsilon>0$ where $-\nu := t - (2-\alpha)(1-t)$ with $t$ defined by the equation $\frac{1}{p} = \frac{t}{p_0} + \frac{1-t}{2}$. 
We see that $\nu$ is positive precisely when $t < (2-\alpha)(1-t)$. 
% This is equivalent to $\alpha-2 < t(1+\alpha-2) = t(\alpha-1)$. 
This is equivalent to $t>\frac{\alpha-2}{\alpha-1}$. 
This in turn occurs precisely when 
\[
\frac{1}{p}
< 
\frac{\frac{\alpha-2}{\alpha-1}}{p_0} + \frac{1-\frac{\alpha-2}{\alpha-1}}{2}.
\]
% Rearranging this becomes 
% $\frac{1}{p} < \frac{\alpha-2}{p_0(\alpha-1)} + \frac{1}{2(\alpha-1)} 
% = \frac{1}{\alpha-1}(\frac{\alpha-2}{p_0}+\frac{1}{2}) 
% = \frac{1}{\alpha-1} \cdot \frac{2(\alpha-2)+p_0}{2p_0}$. 
In other words, we require $p > \frac{2p_0(\alpha-1)}{2\alpha-4+p_0}$. 
Taking $p_0$ towards 1 we are allowed all $p$ such that 
$p > \frac{2(\alpha-1)}{2\alpha-3}$, as claimed. 
\end{proof}

%%%%%%%%%%%%%%%%%%%%%%%%%
% \bigskip
% ---
\section{Completing the main terms by M\"obius inversion}\label{section:Mobius}
% ---

We use the approximations to the multipliers $\what(\xi)$ as defined above. However, the point of this section is to rearrange the functions $m_{\la_l,j}$ in terms of the completed  multipliers $\Omega_{\la,j,d}$, which are defined as
\[
\Omega_{\la,j,d}(\xi) 
= 
\sum_{a\in Z_d}\sum_{\bfa\in Z_d^n}F_d(a,\bfa) e(-\la a/d) \zeta(10^j(\xi-\bfa/d))\ds(\la^{1/k}(\xi-\bfa/d)).
\]
 This is done by an application of M\"obius inversion, giving a divisor sum which leads us to the constants 
\[
C_{j}(d) 
:= 
\sum_{h=1}^\infty\mu(h)\1_{I_j}(dh)
\] 
with $\mu$ the M\"obius function. 
By `completed' we are referring to the sum over $a$ in $\Omega_{\la,j,d}$ which is over the complete residue class $Z_d$ as opposed to $U_d$ in $m_{\la,j}$. 

These constants are bounded in terms of simple functions of $j$ and $d$\,: 
\begin{equation}
|C_{j}(d)| 
\leq 
(2^j/d) \,\1_{I_j}(d).
\end{equation}
To see this note that the bound on $|C_{j,d}|$ follows as 
\[
C_{j}(d) 
= 
\sum_{h=1}^\infty \mu(h)\1_{I_j}(dh)
\leq 
\left( \sum_{h\leq 2^j/d}1 \right) \1_{I_j}.
\]

To carry out the M\"obius inversion argument we need to consider a slightly  more  general class of exponential sums.  Set  \[
F(a,q,\bfa,\bfq)=L^{-n}\sum_{s\in Z^n_L}e(\p(s)a/q+s\cdot\bfa/\bfq),\]
where $\bfq=(q_1,...,q_n)$, $\bfa/\bfq=(a_1/q_1,...,a_n/q_n)$, and $L:=lcm(q,q_1,...,q_n)$ is the least common multiple of the set $\{q,q_1,...,q_n\}$. 
Notice that $F(a,q,\bfa,(q,...,q))=F_q(a,\bfa)$, so the case when $q=q_1=...=q_n$ simply gives us back our previously defined exponential sums.

An important observation concerning the exponential sums $F(a,q,\bfa,\bfq)$ is the following.
 
\begin{lemma}\label{3.1}
Let $q\geq 1$ be a given integer. If $ q_i$ does not divide $q$ for some $1\leq i\leq n$, then for any fixed $a\in U_q$ and any $\bfa\in U_{\bfq}$ we have
 \[F(a,q,\bfa,\bfq)=0.\]
\end{lemma}

\begin{proof}
Fix $q$, $\bfq$, $\bfa\in U_\bfq$, and $a\in Z_q$, $a\not=0$. 
Assume, without loss of generality, that $q_1$ does not divide $q$. 
Then we can write $q=pd$ and $q_1=p_1d$ where the greatest common divisor of $p \geq 1$ and $p_1 > 1$, denoted $\gcd(p,p_1)$, is $1$. 
%Also, $p_1>1$. 

Recall that $L$ is the least common multiple of all $q_1,\dots,q_n$ and $q$ so that 
\begin{align*}
\sum_{s\in Z_L^n} & e(\p(s)a/q+s\cdot \bfa/\bfq)
\\ & = 
\sum_{s_2,...,s_n\in Z_L} \left( \sum_{s_1\in Z_L} e(\p(s_1,...,s_n)a/q+s_1 a_1/q_1) \right) e(a_2 s_2/q_2+...+a_ns_n/q_n).
\end{align*}
Let $Q_1 = pp_1d$ denote the least common multiple of $q$ and $q_1$. 
The inner sum is a multiple of 
\[
\sum_{s_1\in Z_{Q_1}}e(\p(s_1,...,s_n)a/q+s_1 a_1/q_1)
\]
since $Q_1|L$ and the phase is periodic modulo $Q_1$. 
The sum over $s_1\in Z_{Q_1}$ can be written as a sum of $r+qt$ over $r\in Z_q$ and $t\in Z_{p_1}$, giving 
\begin{align*}
\sum_{s_1\in Z_{Q_1}}e(\p(s_1,...,s_n)a/q+s_1 a_1/q_1) 
& = \sum_{r\in Z_q}\sum_{t\in Z_{p_1}}e(\p(r+qt,...,s_n)a/q+(r+qt) a_1/q_1)
\\ & = \sum_{r\in Z_q}e(\p(r,s_2,...,s_n)a/q)e(ra_1/q_1)\left(\sum_{t\in Z_{p_1}}e(pt a_1/p_1)\right).
\end{align*}
The result follows as 
\[
\sum_{t\in Z_{p_1}}e(qt a_1/q_1) 
= 
\sum_{t\in Z_{p_1}}e(pt a_1/p_1)=0
\]
because $(a_1p,p_1)=1$.
\end{proof}

We now come to the main result of this section. 
\begin{lemma}
For a fixed integer $q\geq1$ we have that 
\eq\label{3.2}
m_{\la_l,j,q}(\xi)=\sum_{d|q}\mu(q/d) \Omega_{\la_l,j,d}(\xi),\ee
which in turn gives 
\[
m_{\la_l,j}(\xi) 
= 
\sum_{d=1}^\infty C_{j}(d)\Omega_{\la_l,j,d}(\xi) .
\]
\end{lemma}

\begin{proof}
Fix a generic  $\la \in \Gamma$ and consider $m_{\la,q,j}(\xi)$ for a fixed $q$. 
By the M\"obius inversion formula  the first claim is equivalent to 
\[
\sum_{d|q} m_{\la,j,d}(\xi)=\Omega_{\la,j,q}(\xi).
\]
Hence we consider
\[
\sum_{d|q}\sum_{a \in U_d}\sum_{\bfa\in Z_d^n}F_d(a,\bfa) e(-a \la/d)\zeta(10^j(\xi-\bfa/d))\ds(\la^{1/k}(\xi-\bfa/d)).
\]
By reducing fractions this is simply 
\[
\sum_{d|q}\sum_{a \in U_d}\sum_{\mathbf{d}|d}\sum_{\bfa\in U_d^n}F(a,d,\bfa,\mathbf{d}) e(-a \la/d)\zeta(10^j(\xi-\bfa/\mathbf{d}))\ds(\la^{1/k}(\xi-\bfa/\mathbf{d})).
\]
Here $\mathbf{d}=(d_1,...,d_n)$ and by $\mathbf{d}|d$ we mean that $d_i|d$ for each $i$. Applying Lemma \ref{3.1} lets us write this as  \[
\sum_{d|q}\sum_{a \in U_d}\sum_{\mathbf{d}|q}\sum_{\bfa\in U_q^n}F(a,d,\bfa,\mathbf{d}) e(-a \la/d)\zeta(10^j(\xi-\bfa/\mathbf{d}))\ds(\la^{1/k}(\xi-\bfa/\mathbf{d})),\]
as each of the terms with $\mathbf{d}|q$ and $d_i \nmid \, d$ for some $i$ do not contribute anything. This last expression is  \[
\sum_{a \in Z_q}\sum_{\bfa\in Z_q^n}F(a,q,\bfa,\mathbf{q}) e(-a \la/q)\zeta(10^j(\xi-\bfa/\mathbf{q}))\ds(\la^{1/k}(\xi-\bfa/\mathbf{q})),\]
which is just\[
\sum_{a \in Z_q}\sum_{\bfa\in Z_q^n}F_q(a,\bfa) e(-a \la/a)\zeta(10^j(\xi-\bfa/\mathbf{q}))\ds(\la^{1/k}(\xi-\bfa/\mathbf{q})).\]
This proves the first claim.

We now continue by summing over $q\in I_j$, getting
\begin{align*}
\sum_{q\in I_j} & \sum_{d|q} \mu(q/d) \Omega_{\la_l,j,d}(\xi)
\\ & = 
\sum_{q=1}^\infty\sum_{d=1}^\infty \1_{I_j}(q) \1_{d|q} \mu(q/d) \Omega_{\la_l,j,d}(\xi)
\\ & = 
\sum_{d=1}^\infty\sum_{q=1}^\infty \1_{I_j}(q) \1_{d|q} \mu(q/d) \Omega_{\la_l,j,d}(\xi)
\\ & = 
\sum_{d=1}^\infty\sum_{h=1}^\infty \1_{I_j}(dh) \mu(h) \Omega_{\la_l,j,d}(\xi)
\\ & =
\sum_{d=1}^\infty C_{j}(d) \Omega_{\la_l,j,d}(\xi),
\end{align*}
which is the second claim. 
\end{proof}

Before we close this section it is worth pointing out a few  corollaries of Lemma \ref{3.1} which will be used shortly. 
In the statement we use the notation $V_{\la}(d)$ for the set $\{ s\in Z_d^n: Q(s)=\la\}$.

\begin{cor}
For a given integer $q\geq 1$ we have the following estimates:
\eq\label{F}
\sum_{d|q}\mu(q/d)d^{1-n}|V_{\la}(d)|=\sum_{d|q}\mu(q/d)\left(d^{1-n}|V_{\la}(d)|-1\right)=\sum_{a\in U_q}F_q(a,0)e(-\la a/q),\ee
\eq\label{U}
q^{1-n}|V_{\la}(q)|=\sum_{d|q}\sum_{a\in U_d}F_d(a,0)e(-\la a/d),\ee
and 
\eq\label{C}
q^{1-n}|V_{\la}(q)|\<1\ee
Also, for $t\in U_q$ we also have
\eq\label{K}
|F_q(t,0)|\<\sup_{r\in Z_q}\left|\sum_{d|q}\mu(q/d)\left(d^{1-n}|V_{r}(d)|-1\right)\right|.\ee
\end{cor}

The statements \eqref{F} and  \eqref{U} occur in standard treatments of the singular series and here they follow from direct evaluation of \eqref{3.2} at $\xi=0$. 
Equation \eqref{C} is also known, but follows easily enough from \eqref{F} using the estimates $|F_q(a,0)|\<_\epsilon q^{-2+\eps}$ for all $\epsilon>0$; this estimate is used in the Section~\ref{section:discrete}. 
The last statement should be known but we have not found it in the literature. 
The last estimates follows by multiplying both sides of \eqref{F} by $e(r\la/q)$,  summing in $\la\in Z_q$, and then using the trivial upper bound on the result from left hand side of \eqref{F}.

%%%%%%%%%%%%%%%%%%%%%%%%%
% ---
\section{The discrete maximal function estimate}\label{section:discrete}
% ---

We begin by stating a required real analytic maximal function estimate. For this, fix an integer $d\geq1$ and define the averages 
\[
T_{\la,r} f(y) 
:= 
\int_{\R^n}\widetilde{f}(\xi)\ds(\la^{1/k}\xi)e(-(y-r)\cdot\xi) \,d\xi
\]
and the tail maximal operators 
\[
T^{\kappa}_* f(y) 
= 
\sup_{\la_l \in \LL : \lambda_l \geq \kappa}\sup_{r\in[d]^n}|T_{\la,r} f(y)|
\]
for smooth, compactly supported functions on $\R^n$ where $[d] := \{0,1,...,d-1\}$. 
We also define the tail maximal operators
\[
T^{\kappa}_{*,r} f(y) 
= \sup_{\la_l \in \LL : \lambda_l \geq \kappa}|T_{\la,r} f(y)|.
\]
Note that 
\[
T^{\kappa}_* f(y) 
= \sup_{r\in[d]^n}T^{\kappa}_{*,r} f(y).
\]
For these operators we have $L^p$ boundedness property. 
\begin{prop}\label{real}
For each $p>1$, there exists a positive constant $C(p)$, independent of $d \in \N$, such that 
\[
\|T_*^{d^{C(p)}}f\|_{L^p(\R^n)} 
\< 
\|f\|_{L^p(\R^n)}
\]
with an implied constant independent of $d$.
\end{prop}

This result is enough for us to deduce the required discrete maximal function inequality, which is the main goal of the current section.

\begin{prop}\label{discrete}
Let $p$ and $d\geq 1$ be given, and fix $j\geq1$. 
If $C(p)>0$ is as stated in Proposition \ref{real}, then 
\[
\| \sup_{\la_l\geq d^{C(p)}}|\mathscr{F}^{-1}( {\Omega_{\la_l,j,d}} \cdot \hat{f})| \|_{\ell^p} 
\< 
||f||_{\ell^p},
\]
with an implied constant independent of $d$.
\end{prop}

\begin{proof}[Proof of Proposition~\ref{discrete} assuming Proposition~\ref{real}]
Fix $p$, $d$, and $j$ and let $C = C(p)$ be given from Proposition~\ref{real}.
Begin by factoring the function $\Omega_{\la_l,j,d}(\xi)$ in the usual way: 
\begin{align*}
\Omega_{\la_l,j,d}(\xi) 
& = \sum_{a\in Z_d}\sum_{\bfa\in Z_d^n}F_d(a,\bfa) e(-\la_l a/d) \zeta(10^j(\xi-\bfa/d))\ds(\la_l^{1/2}(\xi-\bfa/d))
\\ & = 
\left(\sum_{a\in Z_d}\sum_{\bfa\in Z_d^n}F_d(a,\bfa) e(-\la_l a/d) \zeta(d(\xi-\bfa/d)) \right)\left(\sum_{\bfa\in Z_d^n} \zeta(10^j(\xi-\bfa/d)) \ds(\la_{l}^{1/k}(\xi-\bfa/d))\right)
\\ & =: 
v_{\la_l,d}(\xi) \cdot u_{\la_l,d,j}(\xi).
\end{align*}
Then
\begin{align*}
\sup_{\la_l\geq d^C}|\mathscr{F}^{-1}( \Omega_{\la_l,j,d}\widehat{f})|
& = 
\sup_{\la_l\geq d^C}|\mathscr{F}^{-1}(\widehat{f}\, u_{\la_l,d,j}v_{\la_l,d} )|
\\ & = 
\sup_{t\in [d]^n} \sup_{\la_l\geq d^C;\, \la_l\equiv t \mod(d) } |\mathscr{F}^{-1}(\widehat{f}\, u_{\la_l,d,j}v_{t,d} )|.
\end{align*}
For a fixed $t$ we  enlarge the supremum in $\la_l$ by removing the congruence condition to get 
\[
\sup_{\la_l\geq d^C}|\mathscr{F}^{-1}( \Omega_{\la_l,j,d}\widehat{f})|
\leq 
\sup_{t\in [d]^n} \sup_{\la_l\geq d^C }|\mathscr{F}^{-1}(\widehat{f} \,u_{\la_l,d,j}v_{t,d} )|.
\]

The last expression can be written in terms of convolutions for which we have convolution kernels $U_{\la_l,j,d}$ and $V_{t,d}$ associated to the multipliers $u_{\la_l,d,j}$ and $v_{t,d}$ respectively. 
Then 
\begin{align*}
\sup_{\la_l\geq d^C}|\mathscr{F}^{-1}( \Omega_{\la_l,j,d}\widehat{f})|
& \leq 
\sup_{t\in [d]^n}\sup_{\la_{l}\geq d^c}|V_{t,d}*U_{\la_l,j,d}*f  )| 
\\ & \leq \sup_{t\in [d]^n}|V_{t,d}|*(\sup_{\la_{l}\geq d^C}|U_{\la_l,j,d}*f|  )
\\ & = \sup_{t\in [d]^n} |V_{t,d}| * U_{*,j,d}f
\end{align*}
where, with an abuse of notation, $U_{*,j,d}f$ denotes the maximal operator defined by  
\[
\sup_{\la_{l}\geq d^C}|U_{\la_l,j,d}*f|.
\]

We can evaluate the convolution kernel $V_{t,d}$ directly, getting 
\[
V_{t,d}(x) 
= 
d\1_{Q(x)\equiv t \mod(d)} \int_{[-1/2,1/2]^n}\zeta(d\xi)e(-x\cdot\xi)d\xi.
\]
With the bound 
\[
\int_{[-1/2,1/2]^n}\zeta(\xi)e(-x\cdot\xi)d\xi\<\frac{ 1 }{(n^{1/2}+1+|x|)^{n+1}},
\]
we see that
\[
|V_{t,d}(x)| 
\< 
\frac{d\1_{Q(x)\equiv t \mod(d)}(x)}{d^n(n^{1/2}+1+|x|/d)^{n+1}}.
\]
This gives
\[
|V_{t,d}|*U_{*,j,d}f 
\< 
\sum_{x\in\Z^n}\frac{d\1_{Q(x)\equiv t \mod(d)}(x)}{d^n( n^{1/2}+1+|x|/d)^{n+1}}U_{*,j,d}f(y-x).
\]
An application H\"older's inequality show us that 
\begin{align*}
(|V_{t,d}|*U_{*,j,d}f (y))^p
\< 
& \sup_{t\in Z_d}\left( \sum_{x\in\Z^n}\frac{d \1_{Q(x)\equiv t \mod(d)}(x)}{d^n(n^{1/2}+1+|x|/d)^{(n+1)}}\right)^{p-1} 
\\ & \quad \quad \times\left( \sum_{x\in\Z^n}\frac{d \1_{Q(x)\equiv t \mod(d)}(x)}{d^n(n^{1/2}+1+|x|/d)^{(n+1)}}|U_{*,j,d}f(y-x)|^p\right).
\end{align*}

Our goal is now to find a bound which is uniform  in $t$ for the last expression. 
For the first factor we see that this is at most a constant uniformly in $t$. 
Indeed, 
\begin{align*}
\sup_{t\in Z_d} \sum_{x\in\Z^n}\frac{d \1_{Q(x)\equiv t \mod(d)}(x)}{d^n(n^{1/2}+1+|x|/d)^{(n+1)}} 
& = 
\sup_{t\in Z_d} \sum_{l\in\Z^n}\sum_{r\in[d]^n}\frac{d \1_{Q(r)\equiv t \mod(d)}(r)}{d^n(n^{1/2}+1+|dl+r|/d)^{(n+1)}}
\\ & = 
\sup_{t\in Z_d} \sum_{l\in\Z^n}\sum_{r\in[d]^n}\frac{d^{1-n} \1_{Q(r)\equiv t \mod(d)}(r)}{(1+|l|)^{(n+1)}},
\end{align*}
% and 
% \[
% \sum_{r\in[d]^n}d^{1-n} \1_{Q(r)\equiv t \mod(d)}(r)\<1
% \]
which is bounded uniformly of $t$ by \eqref{C}. 
Here we used that $n^{1/2}+1+|dl+r|/d\geq n^{1/2}+1+|l|-|r|/d\geq1+|l|$; this explains the presence of the $n^{1/2}$.

The second factor is treated in a similar manner: 
\[
\sum_{x}\frac{d \1_{Q(x)\equiv t \,(d)}}{d^n(n^{1/2}+1+|x|/d)^{(n+1)}}|U_{*,j,d}f(y-x)|^p
\< 
\sum_{l\in\Z^n}\sum_{r\in[d]^n}\frac{d^{1-n} \1_{Q(r)\equiv t \,(d)}}{(1+|l|)^{(n+1)}}|U_{*,j,d}f(y-dl-r)|^p.
\]
We have 
\begin{align*}
\sum_{r\in[d]^n}d^{1-n} \1_{Q(r)\equiv t \mod(d)}|U_*f(y-dl-r)|^p
& \< \sup_{r\in[d]^n} \1_{Q(r)\equiv t \mod(d)}|U_*f(y-dl-r)|^p
\\ & \leq
\sup_{r\in[d]^n} |U_{*,j,d}f(y-dl-r)|^p
\end{align*}
uniformly in $t$, as the initial expression here is essentially an average. 

Now we sum in $y$:
\begin{align*}
\sum_{y\in\Z^n}(|V_{t,d}|*|U_{*,j,d}f| (y))^p 
& \< 
\sum_{y\in\Z^n}\sum_{l\in\Z^n}\frac{1}{(1+|l|)^{(n+1)}} \sup_{r\in[d]^n} |U_{*,j,d}f(y-dl-r)|^p
\\ & = \sum_{y\in\Z^n}\sum_{l\in\Z^n}\frac{1}{(1+|l|)^{(n+1)}} \sup_{r\in[d]^n} |U_{*,j,d}f(y-r)|^p.
\end{align*}

All that remains is to relate the operators 
\[
\sup_{r\in[d]^n} U_{*,j,d}f(y-r)|
\]
to the those treated in Proposition~\ref{real}. 
Recalling that the convolution kernel $U_{\la,j,d}$ has a Fourier multiplier of 
\[
\sum_{\bfa\in Z_d^n} \zeta(10^j(\xi-\bfa/d)) \ds(\la_{l}^{1/k}(\xi-\bfa/d),
\]
which lets us write 
\[
U_{\la,j,d}*f(y-r)=\int_{\Pi^n} u_{\la,j,d}(\xi)e(r\cdot\xi) \widehat{f}(\xi) e(-y\cdot\xi)\,d\xi.
\]
The transference principle of \cite{MSW} (Section 2) reduces the result to Proposition~\ref{real}.
\end{proof}

%%%%%%%%%%%%%%%%%%%%%%%%%
% ---
\section{A Euclidean maximal function}\label{section:Euclidean}
% ---

In this section we prove Proposition~\ref{real}. 
To do this, we need the following estimate from \cite{Akos}; see Lemma 6 there.
\begin{lemma}\label{akos}
If $Q$ is a $\psi$-regular form of degree $k$ in $n$ variables, then we have the decay estimate 
\[
|\ds(\xi)| 
\lesssim 
(1+|\xi|)^{-K}
\]
where $K := \frac{1}{2} (\frac{\mathcal{B}(Q)-(k-1)2^{k-1}}{2^{k}}-1)>0$.
\end{lemma}

In general this estimate is not sharp, but any positive $K$ suffices for our argument. 
For instance, these bounds are readily improved in diagonal situations such as when $Q(x) := x_1^k+\cdots+x_n^k$ where one may take $K := n/k$.

We now continue with the proof of Proposition \ref{real}.
 
\begin{proof}[Proof of Proposition \ref{real}]
Fix $p \in (1,2)$. 
Let $C'$ be the constant chosen so that interpolating $d^n$ at $L^{1+(p-1)/2}$ and $d^{-C'}$ at $L^2$ gives a constant of $d^0=1$ at $L^p$.  More precisely, define $\theta\in[0,1]$ by \[
p^{-1}=(1-\theta)(1+(p-1)/2)+\theta/2\] 
and select $C'$ by the equation  \[
n(1-\theta)  - C'\theta=0.\]
Note that $C'$ does not depend on $d$.

We use an argument similar to that of Lemma~\ref{errors} of Section~\ref{section:reduction}. 
Here we approximate the operators $T_{\la_l,r}$ simply by the unshifted operator $T_{\la_l,0}$. 
Lemma~\ref{akos} paired with  Theorem~A of \cite{DR86} implies that $T^\kappa_{*,0}$ is bounded on $L^p(\R^n)$ for all $p>1$. 
Similarly, for each fixed $r$, we have that $T^\kappa_{*,r}$ is bounded on $L^p(\R^n)$ for all $p>1$. 
 
At $L^{1+(p-1)/2}$ we have 
\begin{align*}
||\sup_{\la_l\geq \kappa}\sup_{r\in[d]^n}|T_{\la_l,r}-T_{\la_l,0}|\,||_{L^{1+(p-1)/2}\to L^{1+(p-1)/2}}
& \< 
\sum_{r\in[d]^n}||\sup_{\la_l\geq \kappa}|T^{\kappa}_{*,r}||_{L^{1+(p-1)/2}\to L^{1+(p-1)/2}}
\\ & \quad \quad 
+ ||T^{\kappa}_{*,0}||_{L^{1+(p-1)/2}\to L^{1+(p-1)/2}}
\\ & \< 
d^n.
\end{align*}
By our choice of $C'$, it suffices for us to show that  
\[
||\sup_{r\in[d]^n}\sup_{\la_l\in I_j;\, \la_l\geq \kappa}|T_{\la_l,r}-T_{\la_l,0}|\,||_{L^{2}\to L^{2}}\<d^{-C'}2^{-\delta j}
\]
for some $\kappa=d^C$.

%as then we have  \[
%||T^{\kappa}_{*,0}||_{L^{1+(p-1)/2}\to L^{1+(p-1)/2}}\<1\]
%for any $p>1$.

Notice  first that 
\[
T_{\la_l,0} f(y)-T_{\la_l,r} f(y) 
= 
\int_{\R^n}\widetilde{f}(\xi)\ds(\la^{1/k}\xi)(1-e(-r\cdot\xi)) e(-y\cdot\xi)\,d\xi,
\]
so that 
\begin{align*}
W 
& := 
||\sup_{\la_l\in I_j;\, \la_l\geq \kappa} \sup_{r\in[d]^n}  \int_{\R^n}\widetilde{f}(\xi) \ds(\la^{1/k}\xi) (1-e(-r\cdot\xi)+1)  e(-y\cdot\xi)\,d\xi||^2_{L^2}
\\ & \leq 
\sum_{\la_l\in I_j;\, \la_l\geq \kappa} \sum_{r\in[d]^n} || \int_{\R^n}\widetilde{f}(\xi)\ds(\la^{1/k}\xi)(1-e(-r\cdot\xi)) e(-y\cdot\xi)\,d\xi||^2_{L^2}
\\ & \< 
\sup_{\la_l\in I_j;\, \la_l\geq \kappa} \sum_{r\in[d]^n} || \int_{\R^n}\widetilde{f}(\xi)\ds(\la^{1/k}\xi)(1-e(-r\cdot\xi)) e(-y\cdot\xi)\,d\xi||^2_{L^2}
\end{align*}
as $\la_l\in I_j$ is satisfied by $O(1)$ many $\la_l$.

Fix $\delta=K/4$, where $K$ is given in Lemma \ref{akos}, and set   $N_l=\la_l^{1/k}$ for convenience. 
We have 
\begin{align*}
W 
& = 
\sup_{\la_l\in I_j;\, \la_l\geq \kappa} \sum_{r\in[d]^n} || \int_{\R^n}\widetilde{f}(\xi)\ds(N_l\xi)(1-e(-r\cdot\xi)) e(-y\cdot\xi)\,d\xi||^2_{L^2} 
\\ & \leq 
\sup_{\la_l\in I_j;\, \la_l\geq \kappa} \sum_{r\in[d]^n}  \int_{|\xi|\leq N_l^{-\delta} d^{-C'-n-1}}|\widetilde{f}(\xi)\ds(N_l\xi)(1-e(-r\cdot\xi))|^2\,d\xi
\\ & \quad \quad 
+ \sup_{\la_l\in I_j;\, \la_l\geq \kappa}\sum_{r\in[d]^n}  \int_{|\xi|\geq N_l^{-\delta} d^{-C'-n-1}}|\widetilde{f}(\xi)\ds(N_l\xi)(1-e(-r\cdot\xi))|^2\,d\xi .
\end{align*}
By Plancherel's Theorem we have \[
W\<a+b\]
where  \[
a=\sup_{\la_l\in I_j;\, \la_l\geq \kappa} \sum_{r\in[d]^n} || f||_{L^2}\sup_{|\xi|\leq N_l^{-\delta} d^{-C'-n-1}}|\ds(N_l\xi)(1-e(-r\cdot\xi))|\]
and \[
b=\sup_{\la_l\in I_j;\, \la_l\geq \kappa} \sum_{r\in[d]^n}|| f||_{L^2}\sup_{|\xi|\geq N_l^{-\delta} d^{-C'-n-1}}|\ds(N_l\xi)(1-e(-r\cdot\xi))|\]

For $a$ we have $|\xi|\leq N_l^{-\delta} d^{-C'-n-1} $ and we use the estimates $d\sigma(N_l\xi)\<1$, $|1-e(-r\cdot\xi)|\leq |r| \,|\xi|$, and $|r|\< d$ to get 
\[
a 
\leq 
\sup_{\la_l\in I_j;\, \la_l\geq \kappa} \sum_{r\in[d]^n}N_l^{-\delta} d^{-C'-n-1} \,d 
\lesssim 
d^{-C'}2^{-\delta j}.
\]

For $b$ we use the bounds $|(1-e(-r\cdot\xi))|\leq 2$ and 
\[
|\ds(N_l\xi)|\< (1+N_l|\xi|)^{-K}.
\]
As 
\[
(1+N_l|\xi|)^{-K}\<(N_l|\xi|)^{-K},
\]
we get that when $|\xi|\geq N_l^{-\delta}d^{-C'-n-1}$ we have 
\[
|\ds(N_l\xi)| 
\< 
N_l^{-K+\delta} d^{C'+n+1} 
\leq 
N_l^{-K/4} N_l^{\delta-K/4}d^{C'+n+1}N_l^{-K/2} 
= 
N_l^{-\delta} (d^{C'+n+1}/N_l^{K/2}).
\]
And then 
\[
a 
\< 
\sup_{\la_l\in I_j;\, \la_l\geq \kappa} \sum_{r\in[d]^n} N_l^{-\delta} (d^{C'+n+1}/N_l^{K/2}) 
= 
2^{-\delta j} (d^{C'+2n+1}/N_l^{K/2}).
\]

We now choose  $C=4K^{-1}(2C'+2n+1)$,  so that $\la_l\geq d^C$ implies  $(d^{C'+2n+1}/N_l^{K/2})\leq d^{-C'}$. This completes the proof.
\end{proof}

%%%%%%%%%%%%%%%%%%%%%%%%%
% \bigskip
% ---
\section{Proof of Theorem~\ref{theorem:new}}\label{section:interpolation}
% ---

Fix $1<p<2$. 
From Section~\ref{section:Mobius} we apply the triangle inequality to find that 
\[
\|M_{*,j}f\|_{\ell^{p}} 
\leq 
\sum_{d=1}^\infty  |C_{j,d}| \, \| \sup_{\la_l}|\mathscr{F}^{-1}(\Omega_{\la_l,j,d}\hat{f})|\|_{\ell^{p}}.
\]
The right hand side is at most
\begin{align*}
\sum_{d=1}^\infty |C_{j,d}| \, \| \sup_{\la_l\geq d^{C(p)}}| \mathscr{F}^{-1}(\Omega_{\la_l,j,d}\hat{f})| \|_{\ell^{p}} 
& + \sum_{d=1}^\infty |C_{j,d}| \, \| \sup_{\la_l\leq d^{C(p)}}|\mathscr{F}^{-1}(\Omega_{\la_l,j,d}\hat{f})|\, \|_{\ell^{p}}
\\ & =: a+b
\end{align*}
where $C(p)$ is the constant appearing in the statement of Proposition~\ref{real}. 

To treat $a$ we use Proposition~\ref{discrete}, which gives 
\[ 
a\< \sum_{d=1}^\infty  |C_{j,d}|\,||f||_{\ell^{p}}.
\]
Summing we find
\[
a\<||f||_{\ell^{p}}\sum_{d=1}^\infty \1_{d\leq2^j}(2^j/d)\<j 2^{j}||f||_{\ell^{p}}.
\] 

For $b$ we see that 
\[
b\leq\sum_{d=1}^\infty  |C_{j,d}|\, \sum_{\la_l\leq d^{C(p)}}||\mathscr{F}^{-1}(\Omega_{\la_l,j,d}\widehat{f}) ||_{\ell^{p}}
\]
For each fixed $\la_l$, $j$ and $d$ we have 
\[
\|\mathscr{F}^{-1}(\Omega_{\la_l,j,d}\widehat{f}) \|_{\ell^{p}} 
\< 
||f||_{\ell^{p}}.
\]
Then 
\[
b \< ||f||_{\ell^{p}} \sum_{d=1}^\infty  |C_{j,d}| \, \#\{\la_l:\la_l\leq d^{C(p)}\}.
\]
From the lacunary assumption on our sequence we get 
\[
\#\{\la_l:\la_l\leq d^{C(p)}\}\< C(p)\log\,d,
\]
which gives 
\[
b\<||f||_{\ell^{p}} \sum_{d=1}^\infty  |C_{j,d}| \log\,d\< ||f||_{\ell^{p}} \,j\,\sum_{d=1}^\infty  |C_{j,d}| \<j^2 2^{j}||f||_{\ell^{p}}.
\]
This completes the proof of Theorem~\ref{theorem:new}.

%\newpage
% ---
\section{Probabilistic counterexamples}
% ---

% \renewcommand{\vector}[1]{{\bf #1}}
\renewcommand{\vector}[1]{{#1}}

In this section we generalize Zienkiewicz's probabilistic counterexamples to prove Theorem~\ref{theorem:lowerbound}. Throughout this section fix our form $\form \in \Z[x_1,\dots,x_n]$ to be a homogeneous, positive definite form of degree $k \geq 2$ satisfying the Birch rank condition $\mathcal{B}(Q)>(k-1)2^k$. 
We will leave the extension to $\psi$-regular forms to the reader. 

We will use the following asymptotic known as \emph{the Lipschitz principle} in \cite{Davenport:Lipschitz}: 
\begin{equation}\label{equation:Lipschitz}
\# \{ \vector{x} \in \Z^n : \form(\vector{x}) \leq R \} 
= 
C_{\form} R^{n/k} + O_{\form}(R^{n/k-1})
\end{equation}
for all sufficiently large $R \gg 1$. 
The implicit constants in \eqref{equation:Lipschitz} only depend on the form. 
For instance $C_\form$ is the volume of the body $\{ \form(\vector{x}) \leq 1 \}$ while the other implied constant depends on the $n-1$-dimensional volume of $\{ \form(\vector{x}) = 1 \}$. 
This is more precise than the following which was a corollary of the main theorem in \cite{Bi}: 
\[
\# \{ \vector{x} \in \Z^n : \form(\vector{x}) \leq R \} 
\eqsim 
R^{n/k}.
\]
Moreover, We have the following version for congruence classes: for fixed $q$ and $b \in Z_q$
\begin{equation}%\label{equation:Lipschitz}
\# \{ \vector{x} \in \Z^n : \form(\vector{x}) \leq R \; \text{and} \; \vector{x} \equiv \vector{b} \mod q \} 
= 
C_{\form} q^{-n} R^{n/k} + O_{\form}(R^{n/k-1}).
\end{equation}
for all $R>0$ sufficiently large with respect to $q$. 
It is important that $C_\form$ is the same as in \eqref{equation:Lipschitz}. 

Since $\form$ is positive definite and $\{ x \in \R^n : Q(x) - \lambda = 0 \}$ for $\lambda>0$ defines an affine algebraic set, $\{ x \in \R^n : Q(x) - \lambda = 0 \}$ satisfies the conditions stated in \cite{Davenport:Lipschitz} sufficient for \eqref{equation:Lipschitz} to hold asymptotically.
We will use an analogous statement for the hypersurfaces of $\form$. 
% In particular we will show the following. 
\begin{prop}\label{proposition:weakLipschitz}
For each $J \in \N$ there exists an $R_0(J)>0$ sufficiently large such that for each $\vector{b} \in Z_{\oddprime^J}^n$ and $R>R_0$ there exists a $\lambda(\vector{b}) \in [R,2R) \cap \N$ such that 
\begin{equation}\label{asymptotic:Z}
\# \{ \vector{x} \in \Z^n : \form(\vector{x}) = \lambda \; \text{and} \; \vector{x} \equiv \vector{n} \mod \oddprime^J \}
\gtrsim 
\oddprime^{-J(n-1)} \lambda^{n/k-1}.
\end{equation}
\end{prop}

\begin{proof}
Fix a choice of $\vector{b} \in Z_{\oddprime^J}^n$. 
From \eqref{equation:Lipschitz} we have 
\begin{align*}
\oddprime^{-Jn} R^{n/k} 
& \eqsim 
\# \{ \vector{m} \in \Z^n : \vector{m} \equiv \vector{b} \; \text{and} \; \form(\vector{m}) \in [R,2R) \} 
\\ & = 
\sum_{\lambda \in [R,2R), \lambda \equiv \form(\vector{b}) \mod \oddprime^J} \# \{ \vector{m} \in \Z^n : \form(\vector{m}) = \lambda \; \text{and} \; \vector{m} \equiv \vector{b} \mod \oddprime^J \} 
\\ & \leq 
(R/\oddprime^J) \max_{\lambda \in [R,2R), \lambda \equiv \form(\vector{n}) \mod \oddprime^J} \# \{ \vector{m} \in \Z^n : \form(\vector{m}) = \lambda \; \text{and} \; \vector{m} \equiv \vector{n} \mod \oddprime^J \}. 
\end{align*}
Therefore, 
\[
\max_{\lambda \in [R,2R), \lambda \equiv \form(\vector{n}) \mod \oddprime^J} \# \{ \vector{m} \in \Z^n : \form(\vector{m}) = \lambda \; \text{and} \; \vector{m} \equiv \vector{n} \mod \oddprime^J \} 
\gtrsim 
\oddprime^{-J(n-1)} R^{n/k-1}.
\]
\end{proof}

%\begin{proof}[Failed previous attempt at Proof]
%From \cite{Serre} we know that for each $N \in Z_{\oddprime^J}$, we have 
%\begin{equation}
%\# \{ \vector{n} \in Z_{\oddprime^J}^n : \vector{n}^2 \equiv N \mod \oddprime^J \} 
%= 
%\oddprime^{J(d-1)} + O(\oddprime^{Jd/2}).
%\end{equation}
%
%Fix a choice of $\vector{n} \in \{ \vector{n} \in Z_{\oddprime^J}^n : \vector{n}^2 \equiv N \mod \oddprime^J \}$ we have for sufficiently large $R$ that there exists a $\lambda \in [R,2R)$ such that $\lambda \equiv N \mod \oddprime^J$ with 
%\[
%\# \{ \vector{m} \in \Z^n : \vector{m}^2 = \lambda \; \text{and} \; \vector{m} \equiv \vector{n} \mod \oddprime^J \} 
%\gtrsim 
%\oddprime^{J(d-1)} \lambda^{d/k-1}
%\]
%
%\begin{align*}
%\# \{ \vector{m} \in \Z^n : \vector{m}^2 = \lambda \} 
%& = 
%\sum_{\vector{n} \in \{ \vector{n}^2 \equiv \lambda \mod \oddprime^J \}} \# \{ \vector{m} \in \Z^n : \vector{m}^2 = \lambda \; \text{and} \; \vector{m} \equiv \vector{n} \mod \oddprime^J \} 
%\\ & < 
%\sum_{\vector{n} \in \{ \vector{n}^2 \equiv N \mod \oddprime^J \}} \delta \oddprime^{-J(d-1)} \lambda^{d/k-1}
%\\ & \leq 
%C_{d,k} \delta \lambda^{d/k-1}
%\end{align*}
%which is a contradiction taking $\delta$ sufficiently small since we may take $C$ dependent only on the dimension $d$ and degree $k$. 
%\end{proof}

We now give the proof of Theorem~\ref{theorem:lowerbound}. 

\begin{proof}[Proof of Theorem~\ref{theorem:lowerbound}]
Fix $\oddprime$ to be an odd prime. 
We will probabilistically construct our lacunary sequence $\LL \subset \N$ by making use of \eqref{asymptotic:Z}. 
For each $J \in \N$, choose $K(J)$ so that Proposition~\ref{proposition:weakLipschitz} holds with $R_0(J) := 2^{K(J)}$. 
Order the vectors in $Z_{\oddprime^J}^n$ as $\vector{b}^i$ for $i=1, \dots \oddprime^{Jd}$. 
Choose $s_{J,i} := \lambda(\vector{b}^i) \in [2^{K(J)+J^2+i-1},2^{K(J)+J^2+i})$ such that \eqref{asymptotic:Z} holds for $\vector{b}^i$. 
For each $J\in \N$ we now have a finite sequence $\lambda_1, \dots, \lambda_{\oddprime^J}$. 
Order the set of all $s_{J,i}$ lexicographically (e.g. in the first index and then the second index) to form the sequence $\LL$ which we also write as $\{\lambda_j : j \in \N\}$. 
Under this lexicographic ordering if $(J,i) < (J',i')$, then by construction $2s_{J,i} < s_{J',i'}$. 
Consequently $\lambda_{j+1} > 2\lambda_j$ for each $\lambda_j \in \LL$ so that we indeed have a lacunary sequence. 

Fixing $J$ for the moment, for each $T \in \N$, we take our function to be $f_T := {\bf 1}_{(\oddprime^J\Z \cap [-T,T])^n}$. 
For $\vector{x} \in [-T/10^n,T/10^n]^n$, we write $\vector{x} = \oddprime^J\vector{y}+\vector{b}$ for some $\vector{y} \in \Z^n$ and some $\vector{b} \in [\oddprime^J]^n$ which we identify with $Z_{\oddprime^J}^n$. 
Then for $\lambda_j$ compatible with $\vector{b}$ by the above construction, we have 
\[
A_{\lambda_j} f_T(\vector{x})
= 
\lambda_j^{1-\frac{n}{d}} \sum_{\vector{z} \in S_\lambda} {\bf 1}_{(\oddprime^J\Z \cap [-T,T])^n}(\vector{x}-\vector{z})
= 
\lambda_j^{1-\frac{n}{d}} \sum_{\vector{z} \in S_\lambda} {\bf 1}_{(\oddprime^J\Z)^n+\vector{b} \cap [-T,T]^n}(\vector{z})
\]
provided that $T$ is sufficiently large with respect to $\oddprime^J$, e.g. $T > 10^{10n} \oddprime^J$. 
The right hand side of the above is 
\[
\lambda_j^{1-\frac{n}{d}} \# \{ \vector{m} \in \Z^n : \form(\vector{m}) = \lambda \; \text{and} \; \vector{m} \equiv \vector{b} \mod \oddprime^J \} 
\gtrsim 
\oddprime^{-J(n-1)}
\]
by Proposition~\ref{proposition:weakLipschitz}. 
Therefore, for $T$ sufficiently large depending on $J$ we have 
\[
\sum_{\vector{x} \in [-T/10^n,T/10^n]} \sup_{j \in \N} |A_{\lambda_j} f_T(\vector{x})|^p 
\gtrsim 
\sum_{\vector{x} \in [-T/10^n,T/10^n]} \oddprime^{-J(n-1)p} 
= 
T^n \oddprime^{-J(n-1)p} .
\]
Comparing this to $\|f_T\|_{\ell^p(\Z^n)}^p \eqsim (T/\oddprime^J)^n$, we see that if $p < \frac{n}{n-1}$, then the $\ell^p(\Z^n)$ operator norm for the lacunary maximal function satisfies the lower bound
\[
\| \sup_{\lambda_j \in \LL} |A_{\lambda_j}| \|_{\ell^p \to \ell^p} 
\gtrsim 
\oddprime^{\frac{-J(n-1)p+Jn}{p}} 
= 
\oddprime^{J(\frac{n}{p}-[n-1])}.
\] 
This is unbounded on $\ell^p(\Z^n)$ for $p<\frac{n}{n-1}$ as $J$ goes to infinity. 
\end{proof}

\begin{rem}
It is important to note that in the above analysis one may instead choose a sequence which grows faster than lacunary, e.g. $\lambda_j \in \N$ such that it grows exponentially like $\lambda_{j+1} > \lambda_j^\eta$ for any $\eta>0$, or worse yet super-exponentially like $\lambda_{j+1} > c^{\lambda_j}$ for some $c>1$, for which the associated maximal function is unbounded near $\ell^1(\Z^n)$. 
In particular, the results of \cite{Cook:sparse} and \cite{Hughes:sparse} are non-trivial despite being restricted to such quickly growing sequences. 
\end{rem}

\bibliographystyle{amsalpha} % Or amsplain

\end{document}